\if{
}\fi
\RequirePackage{fix-cm}
\let\accentvec\vec
\documentclass[smallextended]{svjour3}       % onecolumn (second format)

\let\vec\accentvec

\smartqed  % flush right qed marks, e.g. at end of proof
\usepackage[utf8]{inputenc}
\usepackage[T1]{fontenc}
\usepackage{lmodern}
\usepackage{graphicx}
\usepackage{amsmath}
\usepackage{amsfonts}
\usepackage{lipsum}
\usepackage{multirow}
\usepackage{enumitem}
\newcommand{\N}{\mathbb{N}}					% L'ensemble des entiers naturels
\newcommand{\R}{\mathbb{R}}					% Le corps des nombres reels
	
\newcommand{\xb}{\mathbf{x}} 
\newcommand{\kb}{k} 
\newcommand{\ib}{i}
\newcommand{\hb}{h} 
\newcommand{\gb}{\mathbf{g}} 
\newcommand{\gbeps}{\mathbf{g}_\epsilon} 
\newcommand{\qb}{\mathbf{q}}  

\newcommand{\yb}{\mathbf{y}}

\newcommand{\don}{[d]_0^n}

\newcommand{\resp}{resp.~}

\newcommand{\pminQ}{p_{\min, Q}}
\newcommand{\phanr}{p_{\text{han}, \qb}^{(r)}}
\newcommand{\pschr}{p_{\text{sch}, \qb}^{(r)}}
\newcommand{\pputr}{p_{\text{put}, \gb}^{(r)}}
\newcommand{\pputreps}{p_{\text{put}, \gbeps}^{(r)}}

\newcommand{\pmaxQ}{p_{\max, Q}}

\begin{document}

\title{Error Bounds for Polynomial Optimization over the Hypercube using Putinar type Representations}
%\subtitle{Do you have a subtitle?\\ If so, write it here}

\titlerunning{Error Bounds for Polynomial Optimization over the Hypercube using Putinar type Representations}        % if too long for running head

\author{Victor Magron \thanks{The author was partly supported by an award of the Simone and Cino del Duca foundation of Institut de France.}}

%\authorrunning{Short form of author list} % if too long for running head

\institute{V. Magron \at
              LAAS-CNRS 7 avenue du colonel Roche, F-31400 Toulouse France. \\
              Tel.: +33 (5) 61 33 69 49\\
              \email{magron@laas.fr}
}

\date{Received: date / Accepted: date}
% The correct dates will be entered by the editor
\maketitle
\begin{abstract}
Consider the optimization problem $p_{\min, Q}  :=  \min_{\mathbf{x} \in Q} p(\mathbf{x})$, where $p$ is a degree $m$ multivariate polynomial and $Q := [0, 1]^n$ is the hypercube. We provide explicit degree and error bounds for the sums of squares approximations of $p_{\min, Q}$ corresponding to the Positivstellensatz of Putinar.  
Our approach uses Bernstein multivariate approximation of polynomials, following the methodology of De Klerk and Laurent to provide error bounds for Schmüdgen type positivity certificates over the hypercube. We give new bounds for Putinar type representations by relating the quadratic module and the preordering associated with the polynomials $g_i := x_i (1 - x_i), \: i=1,\dots,n$, describing the hypercube $Q$.
%bounds the error between $\pminQ$ and the bound obtained by Lasserre's relaxation of order $r$ when considering the sums of squares relaxation of .
%Here, we explain how to derive some error bounds for Putinar type representations when optimizing multivariate polynomials over the hypercube $Q$.
%Insert your abstract here. Include keywords, PACS and mathematical
%subject classification numbers as needed.
% \PACS{PACS code1 \and PACS code2 \and more}
% \subclass{MSC code1 \and MSC code2 \and more}
\end{abstract}

\section{Introduction}
\label{sec:intro}

Given a multivariate degree $m$ polynomial $p : \R^n \to \R$ and the unit hypercube $Q := [0, 1]^n$, we consider the following minimization problem:
\begin{equation}
\label{eq:pop}
\pminQ  :=  \min_{\xb \in Q} p(\xb) \:.
\end{equation}
One also defines $\pmaxQ := \max_{\xb \in Q} p(\xb)$.
When $p$ is a quadratic polynomial, this problem includes NP-hard optimization problems in graph, such as maximum cut or stable set problems. 
One way to address Problem~\eqref{eq:pop} is to consider the hierarchy of sums of squares relaxations by Lasserre~\cite{DBLP:journals/siamjo/Lasserre01} to get Putinar type representation~\cite{putinar1993positive} of positive polynomials over $Q$. We refer the interested reader to~\cite{Parrilo2003relax,Las2011book} for more details on these relaxations. Other approaches include  representations derived from the Positivstellensätze of Schmüdgen~\cite{Schmudgen1991} and Handelman~\cite{Handelman1988}. The complexity of Putinar (\resp{}Schmüdgen) Positivstellensätz has been investigated by Nie and Schweighofer in~\cite{DBLP:journals/jc/NieS07} (\resp{}Schweighofer in~\cite{DBLP:journals/jc/Schweighofer04}). However, the error bounds for the approximation obtained from these Positivstellensätze involve some constants, which depend on the problem data and are not trivial to estimate in general. 

This work is a followup of~\cite{deKlerk:2010:Error}, in which the authors use Bernstein multivariate approximation to derive some explicit error bounds for Schmüdgen/Handelman type representations. One of the concluding remarks of~\cite{deKlerk:2010:Error} is that some error bounds for Putinar type representation when optimizing quadratic polynomials over the hypercube can be obtained if the so-called ``$C_n = \frac{1}{n (n + 2)}$ conjecture'' is true.
%conjecture that the quadratic module
Assuming that this conjecture holds, we provide new error bounds for Putinar type representations when minimizing multivariate polynomials over $Q$.

Note that even though the present results rely on this assumption in general, the conjecture can be machine-checked for small integers ($r \leq 6$). Hence, this note states an additional explicit error bound for approximating polynomial optimization problems involving a small number of variables.

\subsection{Preliminaries}
\label{subsect:pre}
Let $\R[\xb]$ (resp.~$\R_{d}[\xb]$) denote the ring of real polynomials (resp. of degree at most $d$) in the variables $\xb=(x_1,\ldots,x_n)$,
while $\Sigma[\xb]$ stands for its subset of sums of squares (SOS) of polynomials.
Given a univariate polynomial $q$,  we define $B_d(q) \in \R_d[x]$ as the Bernstein degree $d$ approximation of $q$:
\begin{align*}
\label{def:bernstein_uni}
B_d(q) := \sum_{k = 0}^d q \Bigl(\frac{k}{d}\Bigr) \binom{d}{k}x^k(1 - x)^{d - k} \enspace.
\end{align*}

We note $\don := \{ 0, 1, \dots, d \}^n$.
For each $\kb = (k_1, \dots, k_n) \in \don$, we set $|\kb|:= \sum_{i = 1}^n k_i$ and $\kb! := \prod_{i = 1}^nk_i!$. Then, the $n$-variate Bernstein polynomial $P_{d, \kb}$ is defined as follows:
\[ P_{d, \kb} := \prod_{i = 1}^n \binom{d}{k_i} x_i^{k_i} (1 - x_i)^{d - k_i}  \enspace .\]

Given an $n$-variate polynomial $p = \sum_{\kb} p_{\kb} \xb^\kb$, we note $B_d(p)$ the Bernstein approximation of order $d$ of $p$. The degree $n d$ polynomial $B_d(p)$ is defined by:
\begin{align*}
\label{def:bernstein_multi}
B_d(p) := \sum_{k_1 = 0}^d \dots \sum_{k_n = 0}^d p \Bigl(\frac{k_1}{d}, \dots, \frac{k_n}{d} \Bigr) P_{d, \kb} \enspace.
\end{align*}
For the sequel, we also need to define $L(p) := \max_{\kb} | p_{\kb} | \dfrac{\kb!}{|\kb|!} $.
%Schmüdgen

\subsection{Error bounds for Schmüdgen type representations}
One can describe the hypercube $Q := [0, 1]^n$ with the linear polynomials:
\begin{equation}
\label{eq:linearg}
q_1 := x_1, \: q_2 := 1 - x_1, \dots, q_{2 n-1} := x_n, \: q_{2 n} := 1 - x_n \:,
\end{equation}
and denote by $H_r(\qb)$ (\resp{}$T_r(\qb)$) the $r$-truncated preprime (\resp{}preordering) generated by $q_1, \dots, q_{2n}$:
\begin{align}
\label{eq:preorder}
H_r(\qb) & := \left\lbrace \sum_{\kb \in \N^{2n}} \lambda_k \qb^k : \: \deg (\lambda_k \qb^k) \leq r, \: \lambda_k \geq 0  \right\rbrace \:, \\
T_r(\qb) & := \left\lbrace \sum_{\kb \in \{0, 1\}^{2 n}} \sigma_k \qb^k : \: \deg (\sigma_k \qb^k) \leq r, \: \sigma_k \in \Sigma[\xb]  \right\rbrace \:.
\end{align}
Moreover, let $H(\qb):=\bigcup_{r \in \N}H_r(\qb)$ (\resp{}$T(\qb):=\bigcup_{r \in \N}T_r(\qb)$) stands for the preprime (\resp{}preordering) generated by $q_1, \dots, q_{2n}$.

\begin{theorem}[Handelman~\cite{Handelman1988}]
Let $p \in \R[\xb]$, with $\qb$ as in~\eqref{eq:linearg}. If $p$ is positive on $Q$, then $p$ belongs to the preprime $H(\qb)$.
\end{theorem}

\begin{theorem}[Schmüdgen~\cite{Schmudgen1991}]
Let $p \in \R[\xb]$, with $\qb$ as in~\eqref{eq:linearg}. If $p$ is positive on $Q$, then $p$ belongs to the preordering $T(\qb)$.
\end{theorem}
Now, define $\phanr :=  \sup \{\mu : p - \mu \in H_r(\qb)\}$ and $\pschr :=  \sup \{\mu : p - \mu \in T_r(\qb)\}$.  From the inclusion $H_r(\qb) \subset T_r(\qb) \subset T(\qb)$, one can deduce that $\phanr \leq \pschr \leq \pminQ$.
In~\cite{deKlerk:2010:Error},  DeKlerk and Laurent give explicit tight error bounds for $\phanr$ and $\pschr$:

\begin{theorem}[from Theorem 1.4~\cite{deKlerk:2010:Error}]
Let $Q$ be described by the polynomials $q_1,\dots,q_{2n}$ from~\eqref{eq:linearg} and let $p \in \R_m[\xb]$. For any integer $r \geq m n$, one has:
\[\max \{\pminQ -  \phanr, \pminQ - \pschr\} \leq  \frac{m^3 n^{m+1}}{6} \frac{L(p)}{r}  \:.\]
\end{theorem}
The bound is sharper in the quadratic case $m = 2$. 

\subsection{Error bounds for Putinar type representations}
\label{subsect:main}
%\cite[Theorem 1.4]{deKlerk:2010:Error}
%
%We remind the results of Nie and Schweighofer which analyze the error for the approximation of 
Alternatively, the hypercube $Q := [0, 1]^n$ can be  described with the polynomials 
\begin{equation}
\label{eq:g}
g_1 := x_1 - x_1^2, \dots, g_n := x_n - x_n^2 \:.
\end{equation}
Let $M_n(\gb) := M_n(g_1, \dots, g_n)$ be the $n$-truncated quadratic module associated with $g_i, \: i=1,\dots,n$.
Define $\pputr := \sup \{\mu : p - \mu \in M_r(\gb)\}$ as the lower bound obtained when solving Lasserre relaxation of Problem~\eqref{eq:pop} at order $r$. In~\cite[Sect. 4]{deKlerk:2010:Error}, the authors conjecture the following:

%The following conjecture allows to derive error bounds for Putinar type representations on the hypercube.
\begin{conjecture}[De Klerk, Laurent~\cite{deKlerk:2010:Error}]
\label{conj:boundconj}
For $n$ even, $\prod_{i = 1}^n x_i + \dfrac{1}{n (n + 2)} \in M_n(g_1, \dots, g_n)$.
\end{conjecture}
This conjecture allows to derive error bounds for Putinar type representations on the hypercube. The authors of~\cite[Sect. 4]{deKlerk:2010:Error} provide an explicit error bound for the quadratic case (see Table~\ref{table:main}). 
We also remind the error bound obtained by Nie and Schweighofer~\cite{DBLP:journals/jc/NieS07}:
\begin{theorem}
\label{th:ns07}
Let $p \in \R_m[\xb]$, $\epsilon > 0$ and $S$ be the hypercube $[0, 1 - \epsilon]^n$, described with the polynomials $g_i:=x_i (1 - \epsilon - x_i), \: i=1,\dots,n$. Let us call $M_r(\gbeps)$ the quadratic module associated with $g_1, \dots, g_n$.
Then, there exists a positive constant $c$ such that:
\begin{enumerate}[noitemsep,topsep=0pt,label={(\roman*)}]
\item If $p$ is positive on $S$, then $p \in M_r(\gbeps)$ for some integer $r \leq c \exp \left(   \left(  m^2 n^m  \dfrac{L(p)}{\pminQ} \right)^c  \right) $.
\item For every integer $r \geq c \exp( (2 m^2 n^m)^c )$, $\pminQ - \pputreps \leq  \dfrac{6 m^3 n^{2 m} L(p)}{\sqrt[c]{\log \frac{r}{c}}  }$.
\end{enumerate}
\end{theorem}
\begin{proof}
\label{pr:ns07}
Let define $M := n \frac{[(1 - \epsilon)^2 + 1]^2}{4}$. We prove that $M - \sum_{i=1}^n x_i^2$ belongs to $M_r(\gbeps)$, using the fact that  $M - \sum_{i=1}^n x_i^2 = (1 + (1 - \epsilon)^2) \sum_{i=1}^n g_i(x_i) + \sum_{i=1}^n [  (1 - \epsilon) x -   \frac{(1 - \epsilon)^2 + 1}{2}]^2 $. Thus, the quadratic module $M_r(\gbeps)$ is archimedean. Then, (i) follows from~\cite[Theorem 6]{DBLP:journals/jc/NieS07} and (ii) from~\cite[Theorem 8]{DBLP:journals/jc/NieS07}.\qed
\end{proof}
In this note, we state the following for Putinar type representation:
\begin{theorem}
\label{th:main}
Let $p \in \R_m[\xb]$, $r \geq m$ and the unit hypercube $Q:=[0, 1]^n$ be described with the polynomials $g_i:=x_i (1 - x_i), \: i=1,\dots,n$. Assume that Conjecture~\ref{conj:boundconj} holds. Then, one has:
\begin{enumerate}[noitemsep,topsep=0pt,label={(\roman*)}]
\item If $p$ is positive on $Q$, then $p \in M_r(\gb)$ for some integer $r \leq \left\lceil \exp \left( m^3 n ^{m+1} \dfrac{L(p)}{\pminQ} \right)  \right\rceil$.
\item $\pminQ - \pputr \leq  \dfrac{3}{2} \dfrac{L(p)}{\log_2 r}   \displaystyle\binom{m+1}{3} n^{m + 1} + \dfrac{\pmaxQ - \pminQ}{r + 2}$.
\end{enumerate}
\end{theorem}
Sect.~\ref{sec:proof} is dedicated to the proof of this result.
%This follows from Theorem~\ref{th:error_bound}, that is proved in Sect.~\ref{sec:proof}.

We summarize in Table~\ref{table:main} the known results for the bound parameters (degree and error bounds) obtained for Schmüdgen and Putinar positivity certificates when optimizing polynomials over the hypercube. Note that in the general case, our error bound is sharper than the bound in Theorem~\ref{th:ns07}~(ii), when choosing the constant $c=1$. However, we lose a factor $n m$ with respect to the exponential argument of the degree bound given in Theorem~\ref{th:ns07}~(i).

\renewcommand{\arraystretch}{1.5}
\begin{table}[!htbp]
\begin{center}
\caption{Comparing degree and error bounds of Schmüdgen/Putinar representations when optimizing polynomials over the hypercube}
\begin{tabular}{cc|ccc}
\hline
Sums of squares  &  Bound & Quadratic case~\cite{deKlerk:2010:Error} & General case & Bounds from \cite{DBLP:journals/jc/Schweighofer04,DBLP:journals/jc/NieS07}\\
representation & Parameter& $m = 2$ & $m \geq 1$ & $m \geq 1$ with $c = c' = 1$\\
\hline
\hline
\multirow{2}{*}{Schmüdgen} 
& Degree& $n^2 \frac{L(p)}{\pminQ}$ &  $m^3 n^{m+1} \frac{L(p)}{6 \pminQ}$~\cite{deKlerk:2010:Error} & $m^4 n^{m} \frac{L(p)}{\pminQ}$\\
%& & & \cite{deKlerk:2010:Error} & \\
\cline{2-5}
& Error& $n^2 \frac{L(p)}{r}$& $m^3 n^{m+1}\frac{L(p)}{6r}$  & $m^4 n^{2m} \frac{L(p)}{r}$ \\
& &for $r \geq 2 n$ & for $r \geq m n$ \cite{deKlerk:2010:Error} & for $r \geq m n^m$\\
\hline			
\multirow{2}{*}{Putinar}
& Degree & $\exp \left( 2 n \frac{L(p)}{\pminQ} \right)$ & $\exp \left( m^3 n ^{m+1} \frac{L(p)}{\pminQ} \right)$ & $\exp \left(  m^2 n ^{m} \frac{L(p)}{\pminQ} \right)$\\
\cline{2-5}
& Error  & $n \frac{L(p)}{\log_2 r} + \frac{\pmaxQ - \pminQ}{r + 2}$  
&  $m^3 n^{m+1} \frac{L(p)}{4 \log_2 r} + \frac{\pmaxQ - \pminQ}{r + 2}$ & $6 m^3 n^{2 m} \frac{L(p)}{\log r}$ \\
& &for $r \geq 2^n$ & for $r \geq \max(m, 2^n)$ & for $r \geq \exp (2 m^2 n^m)$ \\
%\hline
\hline
\end{tabular}
\label{table:main}
\end{center}			
\end{table}

\section{Proof of the main result}
\label{sec:proof}

\if{
Let $p$ be a degree $m$ multivariate polynomial. Here, we consider the following problem:
\begin{equation}
\label{eq:pop}
\pminQ  :=  \min_{\xb \in Q} p(\xb) \enspace,
\end{equation}
One also defines $\pmaxQ := \max_{\xb \in Q} p(\xb)$.
}\fi

\begin{lemma}
\label{th:cneven}
For $n$ even, $\prod_{i = 1}^n x_i + C_n \in M_n(g_1, \dots, g_n)$, for some constant $C_n \leq 1$.
\end{lemma}
\begin{proof}
\label{pr:cneven}
See~\cite[ Sect. 4]{deKlerk:2010:Error}.
\qed
\end{proof}

\begin{lemma}
\label{th:cn}
For $n \in \N_0$, $\prod_{i = 1}^n x_i + C_n \in M_{2 \lceil n/2 \rceil}(g_1, \dots, g_n)$, for some constant $C_n \leq 1$.
\end{lemma}
\begin{proof}
\label{pr:cn}
For $n$ even, it comes from Lemma~\ref{th:cneven}. The case $n=1$ is trivial ($C_1 := 0$). Now, suppose that $n$ is odd with $n = 2 l - 1$ for some $l \geq 2$. Define $C_{2 l - 1} := C_{2 l}$. From Lemma~\ref{th:cneven}, one has the following representation:
\[\prod_{i = 1}^{2 l} x_i + C_{2 l} =  \sum_{i = 1}^{2 l} \sigma_i(\xb) g_i(\xb) \enspace, \]
where $\sigma_0, \dots, \sigma_{2 l}$ are sums of squares of polynomials.
By instantiating $x_n$ with 1, one gets the desired result. \qed
\end{proof}

For $n \in \N_0$, let $C_n$ be as in Lemma~\ref{th:cn} and define $C_n' :=  \sum_{i = 1}^{n} C_i$.
The next lemma allows to express any degree $t$ term $-\lambda \xb^\hb (1 - \xb)^\kb + C_t'$ with positive $\lambda$  as $-\lambda + q$ for some $q \in M_{2 \lceil t/2 \rceil} (\gb)$. %relates the Putinar and Handelman modules by

\if{
\begin{lemma}
\label{th:cnaux}
Then, one has:
\[1 - \prod_{i = 1}^n x_i + C_n'  \in M_{2 \lceil n/2 \rceil}(g_1, \dots, g_n) \enspace .\]
\end{lemma}
}\fi

\begin{lemma}
\label{th:handelmanputinar}
For every $\hb, \kb \in \N^n, 1 -\xb^\hb (1 - \xb)^\kb + C_t' \in M_{2 \lceil t/2 \rceil} (g_1, \dots, g_n)$, where $t := |\hb + \kb| = \deg(\xb^\hb (1 - \xb)^\kb)$.
\end{lemma}

\begin{proof}
\label{pr:cnaux}
First, we show by induction on $n$ that 
\begin{align}
\label{eq:handelmanputinaraux}
1 - \prod_{i = 1}^n x_i + C_n'  \in M_{2 \lceil n/2 \rceil}(g_1, \dots, g_n) \enspace .
\end{align}
For the univariate case it comes from $1 - x_1 = (1 - x_1)^2 + g_1 (x_1)$. Now, let consider the multivariate case $n \geq 2$. One has $1 - \prod_{i = 1}^n x_i + C_n' = 1 - \prod_{i = 1}^{n - 1} x_i + C_{n - 1}' + (1 - x_n) \prod_{i = 1}^{n - 1} + C_n$. Then we conclude using induction applied on the term $1 - \prod_{i = 1}^{n - 1} x_i + C_{n - 1}'$ and Lemma~\ref{th:cn} on the term $(1 - x_n) \prod_{i = 1}^{n - 1} x_i + C_n$.

Now, we introduce new variables $\yb := (x_{|h|+1},\dots, x_{t})$ and replace $1 - \xb$ by $\yb$. Using~\eqref{eq:handelmanputinaraux} on the term $1 - \prod_{i = 1}^t x_i + C_t'$ yields the desired result. \qed
\end{proof}

\begin{theorem}
\label{th:putinar_bound}
Let $p \in \R_m[\xb]$. For all $d \in \N_0$,
%$Q = [0, 1]^n$ and $g_1, \dots, g_n$ as in~\ref{eq:g}. For any integer $d \in \N_0$,  
%there exists an integer $r \leq \max (d n, m)$ such that 
$p - B_d(p) + (1 + C_m' + C_m) \frac{L(p)}{d}  \binom{m+1}{3} n^m \in M_{2 \lceil m/2 \rceil}(\gb)$.
\end{theorem}

\begin{proof}
\label{pr:putinar_bound}
Fix an integer $d \in \N_0$. As in ~\cite[Sect. 3.2]{deKlerk:2010:Error}, one can write:
\begin{align*}
p - B_d(p) = - \sum_{|k| \leq m} p_\kb ( \sum_{j = 1}^n q_{\kb, j} (\sum_{i_j = 0}^{k_j} a_{i_j}^{(k_j)} x_j^{i_j}  ) ) \enspace, 
\end{align*}
where each degree $|\kb| - k_j$ polynomial $q_{\kb, j}$ can be written as $\sum_{\hb, \ib} \lambda_{\hb, \ib} \xb^{\hb} (1 - \xb)^\ib$ for some nonnegative coefficients $\lambda_{\hb, \ib}$ summing up to 1.

%Thus, using Corollary~\ref{th:handelmanputinar}, one can write each $q_{k, j}$ as $\sum_{h, i} \lambda_{h, i}( 1 +  C_{k - k_j}') - q_{k, j}'$ for some $q_{k, j}' \in M_{2 \lceil (|k| - k_j)/2 \rceil} (g_1, \dots, g_n)$.

Then, we split the sum depending on the signs of $p_k$ and of $a_{i_j}^{(k_j)}$ to obtain the following decomposition:
\begin{align*}
p - B_d(p) & = \sum_{\kb | p_\kb > 0} p_\kb \Bigl( \sum_{j = 1}^n   \sum_{i_j = 0}^{k_j - 1} a_{i_j}^{(k_j)} (1 - q_{\kb, j} x_j^{i_j}  ) \Bigr)
                + \sum_{\kb | p_\kb < 0} |p_\kb| \Bigl( \sum_{j = 1}^n   |a_{k_j}^{(k_j)}| (1 - q_{\kb, j} x_j^{k_j}  ) \Bigr)\\
                & + \sum_{\kb | p_\kb < 0} |p_\kb| \Bigl( \sum_{j = 1}^n   \sum_{i_j = 0}^{k_j - 1} a_{i_j}^{(k_j - 1)} q_{\kb, j} x_j^{i_j} \Bigr)
                + \sum_{\kb | p_\kb > 0} p_\kb \Bigl( \sum_{j = 1}^n   |a_{k_j}^{(k_j)}| q_{\kb, j} x_j^{k_j}  \Bigr)
                - C(d, p) \enspace,
\end{align*}
where $C(d, p):= \sum_{\kb | p_\kb > 0} p_\kb  \Bigl(\sum_{j = 1}^n   \sum_{i_j = 0}^{k_j - 1} a_{i_j}^{(k_j)}\Bigr)
+ \sum_{\kb | p_\kb < 0} |p_\kb| \Bigl( \sum_{j = 1}^n   |a_{k_j}^{(k_j)}| \Bigr)$.

It follows from Lemma~\ref{th:cn} that each $q_{\kb, j} x_j^{i_j}$ can be expressed as $q - C_m$ for some $q \in M_{2 \lceil m/2 \rceil} (\gb)$. Similarly using Lemma~\ref{th:handelmanputinar}, each term $1 - q_{\kb, j} x_j^{i_j} + C_m'$ lies in the quadratic module $M_{2 \lceil m/2 \rceil} (\gb)$. Hence, there exists some $q' \in M_{2 \lceil m/2 \rceil} (\gb)$ such that:
\begin{align*}
p - B_d(p)  = q' - (1 + C_m' + C_m) \underbrace{\sum_{|\kb| \leq m} |p_\kb| \Bigl( \sum_{j = 1}^n \sum_{i_j = 0}^{k_j} |a_{i_j}^{(k_j)}| \Bigr)}_{:= C'(d, p)} \enspace .
\end{align*}
As in ~\cite[Sect. 3.2]{deKlerk:2010:Error}, one bounds the constant $C'(d, p)$ using the fact that $|a_{i_j}^{(k_j)}|, \sum_{i_j = 0}^{k_j - 1} a_{i_j}^{(k_j)} \leq \frac{1}{d} \binom{k_j}{2}$ and $|p_\kb| \leq L(p)\frac{|\kb|!}{\kb!}$.
Finally, one obtains  $C'(d, p) \leq \frac{L(p)}{d} \binom{m+1}{3} n^m$, the desired result.\qed
\end{proof}

Theorem~\ref{th:putinar_bound} allows us to derive the following error bound for the Lasserre hierarchy of approximations in the multivariate case:
\begin{theorem}
\label{th:error_bound}
Let $p \in \R_m[\xb]$ and define $r := 2^{n d}$ for some $d \in \N_0$ such that $r \geq m$. Assume that Conjecture~\ref{conj:boundconj} holds. Then, we obtain the following error bound for Putinar type representation when minimizing $p$ over the hypercube $Q$:
\[ \pminQ - \pputr \leq  \frac{3}{2} \dfrac{L(p)}{\log_2 r}   \binom{m+1}{3} n^{m + 1} + \dfrac{\pmaxQ - \pminQ}{r + 2} \enspace . \]
\end{theorem}

\begin{proof}
\label{pr:error_bound}
First, we need to show the following, for every $d \in \N_0$ and even $r \geq n d$:
\begin{align}
\label{eq:aux}
B_d(p - \pminQ) + (\pmaxQ - \pminQ) C_r 2^{n d} \in M_r(\gb) \enspace .
\end{align}
%Let $r$ be an even integer such that $r \geq n d$.
Using the property of the $n$-variate Bernstein polynomial $\sum_{\kb \in \don} P_{d, \kb} = 1$, one has:
\[ B_d(p - \pminQ) = \sum_{\kb \in \don} \Bigl( p\Bigl(\frac{\kb}{d}\Bigr) - \pminQ \Bigr) P_{d, \kb} = \sum_{\kb \in \don} \binom{d}{k_i} \Bigl( p\Bigl(\frac{\kb}{d}\Bigr) - \pminQ \Bigr) \prod_{i = 1}^n  x_i^{k_i} (1 - x_i)^{d - k_i} \enspace. \]
As a consequence of Lemma~\ref{th:cn}, each polynomial  $\prod_{i = 1}^n x_i^{k_i} (1 - x_i)^{d - k_i}$ (of degree at most $n d$) can be written $q_{d, \kb} - C_r$ for some $q_{d, \kb} \in M_r(\gb)$.
Moreover, $2^{n d} = \prod_{i = 1}^n \sum_{k_i = 0}^d \binom{d}{k_i} = \sum_{\kb \in \don} \prod_{i = 1}^n \binom{d}{k_i}$. 
Therefore, one has the decomposition:
\[ B_d(p - \pminQ) + (\pmaxQ - \pminQ) C_r 2^{n d} =  \sum_{\kb \in \don} \binom{d}{k_i}  
\Bigl(p\Bigl(\frac{\kb}{d}\Bigr) - \pminQ\Bigr) q_{d, \kb} + C_r (\pmaxQ - p\Bigl(\frac{\kb}{d}\Bigr)\Bigr)
  \enspace .\]
This achieves the proof of~\eqref{eq:aux}.

Now, by combining the result of Theorem~\ref{th:putinar_bound} with~\eqref{eq:aux}, one obtains for each $r \geq \max(n d, m)$:
\[ p - \Bigl (\pminQ - (1 + C_m' + C_m) \frac{L(p)}{d}  \binom{m+1}{3} n^m - (\pmaxQ - \pminQ) C_r 2^{n d}  \Bigr) \in M_r(\gb) \enspace . \]
By setting $r := 2^{n d}$ and under the assumption $C_r = \frac{1}{r(r + 2)}$, one has for each $d \in \N_0$ such that $r \geq m$:
\[ p - \Bigl (\pminQ -  (1 + C_m' + C_m) \dfrac{L(p)}{\log_2 r}   \binom{m+1}{3} n^{m + 1} - \dfrac{\pmaxQ - \pminQ}{r + 2}   \Bigr) \in M_r(\gb) \enspace . \]
Then, note that for even $m\geq 4$, $C_m' := C_2 + \sum_{k = 3}^m C_k' = \frac{3}{8} - \frac{1}{m + 2}$. Thus, for all $m \geq 2,1 + C_m + C_m' \leq \frac{3}{2}$ and one has:
\begin{align}
\label{eq:error_bound}
 p - \Bigl (\pminQ -  \frac{3}{2} \dfrac{L(p)}{\log_2 r}   \binom{m+1}{3} n^{m + 1} - \dfrac{\pmaxQ - \pminQ}{r + 2}   \Bigr) \in M_r(\gb) \enspace ,
\end{align}
which implies the desired result.\qed
\end{proof}

\paragraph{Proof of Theorem~\ref{th:main}}
Here we focus on proving the statement (i) about the degree bound, since (ii) follows directly from Theorem~\ref{th:error_bound}. Assume that $\pminQ > 0$, $n \geq 2$ and $r \geq m \geq 2$. Using the inequality $|p_\kb| \leq L(p)\frac{|\kb|!}{\kb!}$ and the identity $\sum_{|k|=l} \frac{|\kb|!}{\kb!} = n^l$, for $k \in \N^n$, we obtain that $\pmaxQ \leq L(p) \sum_{|k| \leq m} \frac{|\kb|!}{\kb!} = L(p) \frac{n^{m+1}-1}{n - 1} \leq L(p) n^{m+1}$.
Moreover, one easily shows that $\frac{1}{r+2} \leq \frac{1}{4 \log_2 r}$ and  $\binom{m+1}{3} \leq \frac{m^3}{6}$. Hence, one has:
\[
\frac{3}{2} \frac{L(p)}{\log_2 r} \binom{m+1}{3} n^{m + 1} + \frac{\pmaxQ - \pminQ}{r + 2} \leq m^3 n^{m+1} \frac{L(p)}{2 \log_2 r}
\:.
\]
%Next, we 
As a consequence of~\eqref{eq:error_bound}, $p \in M_r(\gb)$ for some integer $r$ such that $\log r  \leq \left\lceil \left( m^3 n ^{m+1} \dfrac{L(p)}{\pminQ} \right)  \right\rceil$, which concludes the proof. \qed

%\bibliographystyle{plain}
%\bibliography{mybib}

\end{document}